\documentclass[11pt,oneside]{amsart}
\usepackage{mystyle2}
\usepackage{verbatim}
\usepackage{geometry}

\makeatletter
\@ifundefined{example}{%
  \theoremstyle{definition}%
  \newtheorem{example}[theorem]{Example}%
}{}
\makeatother

\title[Finite random measurements]{Analytic inverse problems with\\ finitely many random measurements}

\author{Giovanni S.~Alberti}
\address{DIMA, Department of Mathematics, University of Genoa, Via Dodecaneso 35, 16146 Genova, Italy}
\email{giovanni.alberti@unige.it}

\author{Damiano Poletti}
\address{DIMA, Department of Mathematics, University of Genoa, Via Dodecaneso 35, 16146 Genova, Italy}
\email{damiano.poletti@unige.it}

\author{Simone Sanna}
\address{DIMA, Department of Mathematics, University of Genoa, Via Dodecaneso 35, 16146 Genova, Italy}
\email{simone.sanna@edu.unige.it}

\author{Matteo Santacesaria}
\address{DIMA, Department of Mathematics, University of Genoa, Via Dodecaneso 35, 16146 Genova, Italy}
\email{matteo.santacesaria@unige.it}


\hypersetup{
  pdftitle={Analytic inverse problems with finitely many random measurements},
  pdfauthor={Giovanni S. Alberti, Damiano Poletti, Simone Sanna, and Matteo Santacesaria}
}

\begin{document}

\begin{abstract}
While infinite-dimensional inverse problems are traditionally analyzed assuming continuous data, practical applications rely on finitely many discrete measurements. Recent deterministic approaches establish that unknowns belonging to a $d$-dimensional model class can be stably recovered from finitely many measurements. However, for severely ill-posed problems, such as the Calderón problem and inverse scattering, the known constructions may require a number of measurements that is exponential in $d$. We show that random sampling reduces this count dramatically if one asks only for exact identifiability. By exploiting the analytic geometry of the forward maps, we prove that, whenever the infinite-data problem is injective on the model class, $2d+1$ random scalar measurements determine the unknown uniquely, almost surely. Applications are given to the Calderón problem, with both infinite- and finite-dimensional boundary sampling, and to inverse medium scattering from randomly sampled far-field values.
\end{abstract}

\maketitle

\section{Introduction}\label{sec:introduction}

\subsection{Finite data in inverse problems}
An inverse problem seeks to recover an unknown parameter from indirect observations, often through a partial differential equation or, more generally, through a nonlinear map between infinite-dimensional spaces \cite{Isakov2017-qr,kirsch2011introduction}. The continuous model is natural for the analysis, but practical data acquisition produces only finitely many scalar observations. This leads to a basic sample-complexity question:
\begin{center}
\emph{How many scalar measurements are sufficient to identify every admissible parameter?}
\end{center}

We address this question under a finite-dimensional geometric prior. The admissible parameters are assumed to lie in a $d$-dimensional analytic model class inside an ambient Banach space, see \cref{def:model-class}. Embedded real-analytic manifolds are the basic example, see \cref{ex:manifold}, but the notion also covers a number of structured sets that are not manifolds, see \cref{ex:sparse-parameters} and \cref{sec:calderon-moving}. Under this geometric prior, we derive sufficient bounds on the number of measurements needed in order to guarantee exact reconstruction of an arbitrary unknown in the parameter space. 

\subsection{Main results at a glance}\label{sec:main-results-glance}
Recently, it has been shown that finite deterministic measurements can guarantee stable recovery for finite-dimensional unknowns \cite{alberti2022infinite}. Nevertheless, for notoriously ill-posed problems such as the Calder\'on problem \cite{alessandrini1988stable,mandache2001exponential} or inverse scattering \cite{stefanov1990stability}, this deterministic sampling approach may require an exponential number of measurements relative to the dimension of the parameter space, see \cite[Example 4]{alberti2022infinite}.

Our primary contribution is demonstrating that a probabilistic approach to sampling the measurements drastically reduces the required sample complexity for uniqueness. We establish that uniqueness can be achieved almost surely with a number of measurements that is merely linear in the dimension of the parameter space.

When the unknowns form a finite-dimensional analytic manifold and the measurements are indexed by a finite-dimensional analytic manifold as well, this identification principle is a theorem of Sontag \cite[Theorem 2]{sontag2002experiments}, see also \cite[Theorem 1]{sontag1996critical}, obtained there in the context of parameter identification for ordinary differential equations. We extend it in two directions. First, we allow the sampling space to be infinite-dimensional, which is the situation arising in the Calder\'on problem when the boundary excitations are not constrained to a fixed finite-dimensional space. Second, we replace the manifold hypothesis by the notion of analytic model class, along the lines indicated in \cite[Remark 1.1]{sontag2002experiments}. Both extensions are needed for the applications, and do not require an increase of the number of measurements.

In order to highlight the impact of this approach, we first describe the two principal applications of our main result \cref{thm:main} informally. Together with them we also illustrate the flexibility of the notion of analytic model class.
\vspace{.1cm}
\paragraph{\emph{The Calder\'on problem.}}
Let $\Omega\subset\RR^m$, $m\geq 2$, be a bounded Lipschitz domain and let $\mathcal M\subset L^\infty(\Omega)$ be a given $d$-dimensional analytic model class of uniformly positive conductivities. For  $\sigma\in\mathcal M$, the Neumann-to-Dirichlet map $\Lambda_\sigma$ sends a mean-free boundary current $f$ to the corresponding mean-free boundary voltage, see \eqref{eq:ND-map}. Given independent Gaussian boundary functions $f_i,g_i$, $i=1,\dots,M$, we observe only the scalar matrix elements
\begin{equation*}
    \bigl(\langle \Lambda_\sigma f_i,g_i\rangle_{L^2(\partial\Omega)}\bigr)_{i=1}^M.
\end{equation*}
We prove that, if the full forward map $\sigma\mapsto\Lambda_\sigma$ is injective on the fixed model class, then any $M\geq 2d+1$ independently sampled observations of this form identify every conductivity in the class almost surely; see \cref{thm:calderon-gaussian}.

We also prove a finite-dimensional boundary-sampling result. On a compact parameter set, injectivity of the full forward map and of its differential implies that a sufficiently large finite-dimensional truncation of the boundary space preserves both properties. Random currents and voltages drawn from that fixed finite-dimensional space then yield the same $2d+1$ sample bound; see \cref{thm:calderon-finite}.
\vspace{.1cm}
\paragraph{\emph{Inverse scattering.}}
Let $n$ be the sought-after refractive index of an inhomogeneous medium supported within a ball $B\subset\RR^3$. For an incident direction $\theta\in S^2$, we probe the medium using the incident plane wave
\begin{equation*}
    u_\theta^{\rm i}(x)=e^{ikx\cdot\theta}.
\end{equation*}
As this wave interacts with the inhomogeneous medium, it produces a scattered wave. The sum of the incident wave and this scattered one constitutes the total field. Then, detectors are placed far away to record the complex scattering amplitude. These measurements are captured by the far-field pattern $u_n^\infty(\widehat{x},\theta)$, which represents the scattering amplitude in an observation direction $\widehat{x}\in S^2$. It is governed by the integral equation \eqref{eq:far-field}. The corresponding inverse problem consists of recovering $n$ after probing the medium with incident waves in every possible direction $\theta\in S^2$ and measuring the corresponding far-field pattern on the whole $S^2$. Assuming that the full far-field map $n\mapsto u^\infty_n$  is injective on a given $d$-dimensional analytic model class, then $2d+1$ independent random pairs $(\widehat x_i,\theta_i)$ identify every admissible refractive index almost surely through the measurements $u_n^\infty(\widehat x_i,\theta_i)$; see \cref{thm:scattering}.
\vspace{.1cm}
\paragraph{\emph{Analytic model classes.}} The unknowns are not assumed to form a manifold. What the proof uses is that they are described, up to countably many pieces, by at most $d$ real parameters ranging in subanalytic sets, in such a way that the measurements depend analytically on those parameters; this is the notion of analytic model class, see \cref{def:model-class}. Embedded real-analytic manifolds are the basic example, see \cref{ex:manifold}, but the definition also covers the structured priors used in practice, which are typically not manifolds, and it does so at no cost in the number of measurements. In \cref{ex:sparse-parameters} we discuss $s$-sparse unknowns with respect to a fixed dictionary, which are identified almost surely by $2s+1$ random measurements, the sparsity level replacing the ambient dimension, see \cref{cor:sparse}. In \cref{sec:calderon-moving}, we discuss the recovery of piecewise constant conductivities on moving partitions of the domain.

\subsection{Geometric mechanism}
The effectiveness of this approach stems from a geometric intuition. In the examples we presented above, the forward model is analytic. Therefore, the set of observations that do not distinguish between two given parameters is the zero-level set of an analytic function.

If the underlying infinite-dimensional problem is known to be uniquely solvable, this analytic function cannot be identically zero. Consequently, its zero set has positive codimension in the sampling space. To obtain a uniform statement over all parameter pairs, we form an incidence set over the $2d$-dimensional pair manifold and estimate the dimension of its projection onto the space of measurement systems. Once the number of scalar measurements exceeds $2d$, the uninformative measurement systems have positive codimension and hence probability zero.

This mechanism is the identification principle for analytic response maps of \cite{sontag1996critical,sontag2002experiments}. It is also reminiscent of Whitney-type and prevalent embedding theorems
\cite{whitney1936,botvinick2026unique}, a connection discussed in
\cite[\S6.6]{sontag2002experiments}, but the observables in the present work are prescribed by the physics of the inverse problem rather than chosen from an unrestricted class.

\subsection{Related works}\label{sec:related-work}
Finite data uniqueness and stability for inverse problems have been studied in a number of deterministic settings. For the Calder\'on problem and related elliptic inverse problems, see \cite{alberti2019calderon,alberti2022calderon,harrach2019uniqueness,hanke2024lipschitz,ruland2019fractional,ruland2022runge}; for inverse scattering with strongly structured unknowns, see \cite{blaasten2020recovering,blaasten2021corners}. Fully discrete regularization in the Calder\'on and scattering settings is studied in \cite{felisi2024full,di2026discretization}. Generic finite measurements obtained through singularity-theoretic arguments also appear in the fractional Calder\'on problem with drift \cite{cekic2020calderon}. Abstract finite-measurement stability results for finite-dimensional subspaces and manifolds were developed in \cite{alberti2022infinite,alberti2022inverse}. The recent work \cite{carstea2026h} derives H\"older stability and qualitative finite determinacy from exact uniqueness in finite-dimensional analytic inverse problems. Our result is complementary: it gives an explicit, universal almost sure count of random scalar measurements, but does not establish stability. For analytic response maps depending on finitely many parameters and on finitely many experimental variables, the identification principle used in this work, together with the counts $2d+1$ for uniform and $d+1$ for pointwise identifiability, is due to Sontag \cite[Theorem 1]{sontag1996critical}, \cite[Theorem 2]{sontag2002experiments}; the extension to subanalytic parameter sets is indicated in \cite[Remark 1.1]{sontag2002experiments}. Analytic-function arguments of the same type were used, independently, to determine two-dimensional Navier--Stokes flows from finitely many nodal values \cite{friz2001parametrising,robinson2001rigorous}, and generic observability from $2d+1$ samples in control theory goes back to \cite{aeyels1981generic,aeyels1981samples}.

The geometric antecedents include Whitney's embedding theorem \cite{whitney1936}, Ma{\~n}{\'e}'s projection theorem \cite{mane1981dimension}, Takens' delay-coordinate theorem \cite{takens1981detecting}, and the ``embedology'' framework of Sauer, Yorke, and Casdagli \cite{sauer1991embedology}. The appearance of the count $2d+1$ in both settings is not a coincidence, and the connection is discussed in \cite[\S6.6]{sontag2002experiments}: in the embedding theorems the number of observations is dictated by the fact that the pairs of distinct points to be separated form a set of twice the dimension of the unknown, which is precisely the balance underlying \cref{thm:main}. What differs is the mechanism. Embedding theorems of Whitney type let the observations be chosen freely within a linear family, require no analyticity, and discard the bad choices by a transversality argument; the measurements considered here are prescribed by the forward map of the inverse problem and can be neither perturbed nor taken linear, and it is the real-analytic structure of that map that makes the uninformative measurements negligible. The same substitution is what allows the unknowns to form a subanalytic set rather than a manifold, as in \cref{cor:sparse}, a situation that transversality arguments do not reach. Whitney and Takens embedding arguments have recently been used in \cite{botvinick2026unique} to show that finitely many measure-valued input--output pairs determine a transport map or a vector field. Generic and prevalent finite-dimensional projections of compact or fractal subsets of infinite-dimensional spaces, the infinite-dimensional counterpart of Whitney's theorem, were studied in \cite{hunt1992prevalence,hunt1999regularity,robinson2009linear}, and random stable embeddings of smooth manifolds and of more general structured sets in \cite{baraniuk2009random,dirksen2016dimensionality}.

Our incidence set argument is also related to algebraic identifiability in phase retrieval and algebraic compressed sensing, where the dimension of a bad measurement variety controls generic injectivity \cite{conca2015algebraic,breiding2021algebraic}. In signal processing, random measurements of sparse sets, manifolds, and generative-model ranges lead to stable recovery under quantitative geometric assumptions \cite{candes2006robust,donoho2006compressed,Foucart2013-od,bora2017compressed}. These principles have also been applied to inverse scattering \cite{gilbert-levinson-schotland-2020,alberti-petit-santacesaria}. The present theorem instead concerns exact global injectivity for structured nonlinear measurements and uses real-analytic geometry rather than concentration inequalities.

Finally, randomization is increasingly used for data selection and sketching in PDE-constrained inverse problems. We refer to \cite{jin2024unique,HELLMUTH2026} and the references therein. Statistical and Bayesian inverse problems pursue a different objective: the unknown is assigned a prior distribution and uncertainty is quantified through a posterior law \cite{kaipio2005statistical,stuart2010inverse}; see \cite{abraham2020statistical} for the statistical Calder\'on problem.

\subsection{Open questions and comments}

The count $2d+1$ obtained above concerns scalar observations rather than
physical experiments. In the Calder\'on problem, one applied current $f$
produces the whole boundary voltage $\Lambda_\sigma f$, while in inverse
scattering one incident direction $\theta$ produces the full far-field pattern
$u_n^\infty(\cdot,\theta)$. It is therefore natural to ask whether these
vector-valued data permit recovery from fewer experiments, possibly from a
single generic input--response pair. Experiments whose outcome is a whole
measure are considered in \cite{botvinick2026unique}, where Whitney and Takens
embedding arguments show that finitely many measure-valued input--output pairs
identify a transport map or a vector field. In
\cite{boulle2026zero}, the authors prove a sharp zero--one law of this type for linearly
parameterized analytic systems: either no input identifies the coefficients,
or almost every input sampled from a nondegenerate Gaussian measure does.
Extending this dichotomy to nonlinear parameter-to-operator maps may reveal
if one-shot recovery in the Calder\'on and
scattering problems is possible.

The present results concern uniqueness only. The approach of
 \cite{carstea2026h}, based on {\L}ojasiewicz-type inequalities,
suggests that, for almost every fixed sampling system, analytic injectivity
on a compact model class may yield a sample-dependent H\"older estimate. A
quantitative version should, for every $\varepsilon>0$, provide deterministic
constants $C_\varepsilon>0$ and exponents
$\alpha_\varepsilon\in(0,1]$ for which a uniform H\"older estimate holds with
probability at least $1-\varepsilon$. The stability constant should be
expected to deteriorate as $\varepsilon\downarrow0$. Lipschitz stability,
corresponding to $\alpha_\varepsilon=1$, may be possible under stronger
uniform lower bounds for the differential of the sampled forward map.

Reconstruction is another open direction. In the sparse setting of
\cref{cor:sparse}, one could adapt nonlinear restricted-isometry and
iterative hard-thresholding techniques \cite{blumensath2013compressed}, as
well as sparsity-promoting, support-informed proximal methods for nonlinear
EIT \cite{lazzaro2024oracle}. The main challenge is to derive robust
geometric conditions for the present random measurements that imply global
convergence, possibly from arbitrary initializations, and stability with
respect to noise.

Finally, it would be interesting to extend the incidence-set argument to
other nonlinear inverse problems, especially those governed by semilinear
or quasilinear PDEs. The key requirements are analytic dependence on the
unknown and the experimental input, together with a suitable full-data
separation property.

\subsection{Organization of the paper}
The abstract framework is introduced in \cref{sec:framework}. After recalling the notions from real-analytic geometry that we need (\cref{sec:prel}), we describe the two measurement architectures, the first indexed by a finite-dimensional analytic manifold (\cref{sec:finite-analytic-sampling}) and the second consisting of multilinear measurements generated by a nondegenerate Gaussian measure on a Hilbert space (\cref{sec:gaussian-sampling}), we define analytic model classes (\cref{sec:model-classes}), and we state the main abstract theorem, \cref{thm:main}. The applications to the Calder\'on problem and inverse scattering are stated in \cref{sec:applications}. The geometric proof of the abstract results is given in \cref{sec:proof-main}, the application to the Calder\'on problem in \cref{sec:proof-calderon}, and the application to inverse scattering in \cref{sec:proof-scattering}.

\section{Random analytic measurements}\label{sec:framework}

Throughout the paper, $X$ is a real Banach space with norm $\|\cdot\|_X$, $A\subset X$ is open, and $\calM\subset A$ is an analytic model class of finite dimension $d$, in the sense of \cref{def:model-class} below. Complex Banach spaces are always regarded as real Banach spaces when analyticity with respect to the parameter is discussed. The scalar field is denoted by $\KK\in\{\RR,\CC\}$.

\subsection{Preliminaries}\label{sec:prel}

We first make precise the two notions of analyticity used throughout.

\begin{definition}\label[definition]{def:analytic}
Let $Y_1,Y_2$ be real Banach spaces and let $U\subset Y_1$ be open. A map
$G\colon U\to Y_2$ is \emph{real analytic} if for every $y\in U$ there exist $\rho>0$ and
continuous symmetric $j$-linear maps $G_j\colon Y_1^j\to Y_2$, $j\geq0$, with
$\sum_{j\geq0}\norm{G_j}\rho^j<\infty$ and
\[
    G(y+h)=\sum_{j\geq0}G_j(h,\ldots,h),
    \qquad \norm{h}_{Y_1}<\rho .
\]
A subset $\calM\subset X$ is an \emph{embedded real-analytic manifold of dimension $d$} if
for every $a\in\calM$ there are an open set $V\subset\RR^d$ and a real-analytic map
$\varphi\colon V\to X$ such that $d\varphi_t\colon\RR^d\to X$ is injective for every
$t\in V$ and $\varphi$ is a homeomorphism from $V$ onto an open neighbourhood of $a$ in
$\calM$, the latter being endowed with the topology induced by $X$. We call such a
$\varphi$ a \emph{chart} of $\calM$ and we set $T_a\calM=d\varphi_{\varphi^{-1}(a)}(\RR^d)$.
\end{definition}

The transition map between two charts is real analytic. Indeed, if $\varphi,\psi$ are
charts around $a$ and $s=\psi^{-1}(a)$, then $d\psi_s$ has finite-dimensional, hence
complemented, range, so there is a bounded linear map $L\colon X\to\RR^d$ for which
$L\circ d\psi_s$ is invertible; the analytic inverse function theorem applied to
$L\circ\psi$ shows that $\psi^{-1}=(L\circ\psi)^{-1}\circ L$ is real analytic near $a$,
and therefore so is $\psi^{-1}\circ\varphi$. Thus $\calM$ carries a well-defined
$d$-dimensional real-analytic structure, and the inclusion $\calM\hookrightarrow X$ is
real analytic in every chart.

We now recall the notions from real-analytic geometry that we shall use, referring to
\cite[\S2 and \S3]{bierstone1988semianalytic} and \cite{narasimhan1985analysis} for a
systematic treatment.

Let $\mathcal N$ be a real-analytic manifold, that is a manifold whose transition maps between charts are real analytic functions. A subset $B\subset\mathcal N$ is
\emph{semianalytic} if every point of $\mathcal N$ has a neighbourhood $U$ such that
\[
    B\cap U=\bigcup_{i=1}^{p}\bigcap_{j=1}^{q}X_{ij},
\]
where each $X_{ij}$ is of the form $\{x\in U:f_{ij}(x)=0\}$ or $\{x\in U:f_{ij}(x)>0\}$ for
some real-analytic function $f_{ij}\colon U\to\RR$; in other words, $B$ is locally described
by finitely many analytic equations and inequalities. In particular, the zero set of a
real-analytic function is semianalytic. A subset $B\subset\mathcal N$ is \emph{subanalytic}
if every point of $\mathcal N$ has a neighbourhood $U$ such that $B\cap U=\pi(B')$ for some
real-analytic manifold $\mathcal N'$ and some relatively compact semianalytic set
$B'\subset\mathcal N\times\mathcal N'$, where $\pi\colon\mathcal N\times\mathcal N'\to\mathcal N$
denotes the projection. Every semianalytic set is subanalytic, but the projection of a
semianalytic set need not be semianalytic; this is why \cref{lem:fiber-dimension} yields a
subanalytic set.

Both classes admit locally finite stratifications into real-analytic submanifolds, and we use the standard notion of dimension of a semianalytic or subanalytic set, defined as the largest dimension of a stratum in an analytic stratification \cite[Remark 2.12]{bierstone1988semianalytic}. In particular, since on smooth manifolds this agrees with Hausdorff dimension, this notion of dimension coincides with the Hausdorff one, and it is monotone with respect to inclusion. Subanalyticity of a subset of an open set $U\subset\RR^n$
is always understood relative to $U$. In particular, $U$ is semianalytic in itself and has dimension $n$. We shall use repeatedly that finite intersections of subanalytic sets are subanalytic, that the dimension of a product is the sum of the dimensions, and that the closure of a bounded subanalytic set is subanalytic of the same dimension; see \cite[\S3]{bierstone1988semianalytic} and \cite[Appendix A]{sontag1996critical}.
\smallskip

We consider two sampling architectures.

\subsection{Finite-dimensional analytic sampling}\label{sec:finite-analytic-sampling}
Let $Z$ be a connected, second-countable real-analytic manifold of finite dimension $q\geq1$, endowed with a smooth positive volume measure $\vol_Z$.
Let $\mu$ be a Borel probability measure such that
\begin{equation}\label{eq:mu-ac}
    \mu\ll \vol_Z.
\end{equation}
Let
\begin{equation}\label{eq:analytic-measurement-map}
    \mathfrak m\colon A\times Z\longrightarrow\KK
\end{equation}
be a real-analytic measurement map. We impose that any parameter $a\in\calM$ is uniquely determined when the measurements associated with every possible sample in $Z$ are observed. In other words, we assume
\begin{equation}\label{eq:analytic-separation}
    \mathfrak m(a_1,\cdot)-\mathfrak m(a_2,\cdot)\not\equiv0\ \text{on }Z,\quad\text{for every}~a_1\neq a_2\ \text{in }\calM.
\end{equation}

This formulation covers pointwise sampling of analytic data, as in inverse scattering, and scalar matrix elements restricted to a finite-dimensional input space, as in \cref{thm:calderon-finite}.

\subsection{Gaussian multilinear sampling}\label{sec:gaussian-sampling}
Let $H$ be a separable Hilbert space over $\KK$, let $Z$ be an infinite-dimensional separable real Hilbert space, and let $r\geq1$. We write $\calL^r(H;\KK)$ for the Banach space of bounded $\KK$-multilinear forms on $H^r$. Let the forward map
\begin{equation}\label{eq:multilinear-forward}
    F\colon A\longrightarrow\calL^r(H;\KK)
\end{equation}
be injective on $\calM$. Let $\psi\colon Z\to H$ be bounded and linear, and assume that
\begin{equation}\label{eq:dense-span}
    \overline{\operatorname{span}_{\KK}\psi(Z)}=H.
\end{equation}
The scalar measurement map is
\begin{equation}\label{eq:gaussian-measurement-map}
    \mathfrak m\colon A\times Z^r\to \KK,\qquad \mathfrak m(a,z)
    =F(a)\bigl(\psi z^1,\ldots,\psi z^r\bigr).
\end{equation}
Since $\psi$ is only real linear, $\mathfrak m(a,\cdot)$ is $\RR$-multilinear on $Z^r$ even when $\KK=\CC$.
We abuse notation and we denote the measurement map by $\mathfrak m$ both for the finite-dimensional analytic sampling (eq.\ \eqref{eq:analytic-measurement-map}) and for the Gaussian multilinear sampling (eq.\ \eqref{eq:gaussian-measurement-map}).

Let $\mu$ be a nondegenerate Gaussian Borel probability measure on $Z$. This means that its covariance operator has trivial kernel. Equivalently, every nonzero continuous linear functional on $Z$ has a nondegenerate one-dimensional Gaussian distribution under $\mu$ \cite{Bogachev1998-oc,da2006introduction}. One sample is an $r$-tuple $z=(z^1,\ldots,z^r)\in Z^r$, distributed according to $\mu^r$.

The injectivity of $F$ on $\calM$ and \eqref{eq:dense-span} imply the analogue of the point-separation condition \eqref{eq:analytic-separation}, namely
\begin{equation}\label{eq:gaussian-separation}
    \mathfrak m(a_1,\cdot)-\mathfrak m(a_2,\cdot)\not\equiv0\ \text{on }Z^r,\qquad\text{for every}~a_1\neq a_2\ \text{in }\calM.
\end{equation}
Indeed, assume that $\mathfrak m(a_1,z)=\mathfrak m(a_2,z)$ for every $z\in Z^r$ and set
$T=F(a_1)-F(a_2)\in\calL^r(H;\KK)$. By \eqref{eq:gaussian-measurement-map}, $T$ vanishes on
$\psi(Z)^r$; by $\KK$-multilinearity it vanishes on
$\bigl(\operatorname{span}_\KK\psi(Z)\bigr)^r$, and by continuity and \eqref{eq:dense-span}
it vanishes on $H^r$. Hence $F(a_1)=F(a_2)$, and the injectivity of $F$ on $\calM$
gives $a_1=a_2$.

For either architecture, denote the sample space and probability measure by $(\calZ,\nu)$: in the finite-dimensional case, $(\calZ,\nu)=(Z,\mu)$ and the measurement is \eqref{eq:analytic-measurement-map}; in the Gaussian case, $(\calZ,\nu)=(Z^r,\mu^r)$ and the measurement is \eqref{eq:gaussian-measurement-map}.

\subsection{Analytic model classes}\label{sec:model-classes}
We now make precise the assumption on the set of unknowns. The requirement is that its elements be described, up to countably many pieces, by at most $d$ real parameters, in such a way that the measurements depend analytically on those parameters.

\begin{definition}\label[definition]{def:model-class}
A set $\calM\subset A$ is an \emph{analytic model class of dimension $d$} if there are countably many open sets $U_j\subset\RR^{n_j}$, subanalytic sets $S_j\subset U_j$ with $\dim S_j\leq d$, and maps $\Psi_j\colon S_j\to A$, such that
\begin{equation}\label{eq:model-class}
    \calM=\bigcup_{j\geq1}\Psi_j(S_j),
\end{equation}
and such that, for every $j$, the measurements pulled back through $\Psi_j$ admit a real-analytic extension to $U_j$, namely
\begin{enumerate}[label=\textup{(\roman*)},leftmargin=*]
    \item in the finite-dimensional analytic sampling of \cref{sec:finite-analytic-sampling}, there is a real-analytic map
    \[
      \mathfrak m_j\colon U_j\times Z\to\KK,\qquad   \mathfrak m_j(t,z)=\mathfrak m\bigl(\Psi_j(t),z\bigr),\quad t\in S_j,
    \]
    is real analytic;\label{i}
    \item in the Gaussian multilinear sampling of \cref{sec:gaussian-sampling}, there is a real-analytic map
    \[
      F_j\colon U_j\to\calL^r(H;\KK),\qquad   F_j(t)=F\bigl(\Psi_j(t)\bigr),\quad t\in S_j,
    \]
    is real analytic.
    \label{ii}
\end{enumerate}
In case \ref{ii} we shall write $\mathfrak m_j(t,z)=F_j(t)(\psi z^1,\ldots,\psi z^r)$, so that in both cases $\mathfrak m_j$ is real analytic on $U_j\times\calZ$ and
\[
    \mathfrak m_j(t,z)=\mathfrak m\bigl(\Psi_j(t),z\bigr),\qquad t\in S_j,\ z\in\calZ .
\]
We call $(U_j,S_j,\Psi_j)_{j\geq1}$ an \emph{analytic parametrization} of $\calM$.
\end{definition}

\begin{remark}\label[remark]{rem:model-class-analytic-psi}
Conditions \ref{i} and \ref{ii} are automatically satisfied whenever $\Psi_j$ is the restriction to $S_j$ of a real-analytic map $U_j\to X$ with values in $A$, provided that, in the Gaussian case, $F$ is additionally real analytic on $A$. Indeed, the pulled-back measurements are then compositions of real-analytic maps in the sense of \cref{def:analytic}. This is the case in \cref{ex:manifold} and \cref{ex:sparse-parameters}, and the reader may take it as the definition on a first reading. The more flexible formulation is needed for coefficients whose dependence on the parameters is not continuous in the norm of $X$, as, for instance, could be the case for a partition of the domain whose interfaces move with the parameters.

The maps $\Psi_j$ are not assumed to be injective, and their ranges may overlap. Accordingly, the conclusions below concern the elements of $\calM$ and not the parameters describing them. Within a fixed parametrization $\Psi_j$, the parameter is also identified whenever $\Psi_j$ is injective.
\end{remark}

We collect here two examples of analytic model classes. In each of them the measurement counts are conditional on the separation hypotheses \eqref{eq:analytic-separation} and \eqref{eq:gaussian-separation}, which is a statement about the specific inverse problem, as in \cref{rem:calderon-injectivity,rem:scattering-injectivity}. For these two examples, in the Gaussian case we assume additionally that $F$ is real analytic on $A$.

\begin{example}[Real analytic manifolds]\label[example]{ex:manifold}
The basic example, and the one used in \cref{thm:calderon-finite}, is a finite-dimensional manifold of coefficients.

In mathematical terms, let $\mathcal N\subset X$ be a second-countable embedded real-analytic manifold of dimension $d$, in the sense of \cref{def:analytic}. Then $A\cap\mathcal N$ is an analytic model class of dimension $d$. Indeed, second countability provides a countable analytic atlas $(V_j\subset\RR^d, \varphi_j\colon V_j\to X)_{j\geq1}$ of $\mathcal N$, and it suffices to take
\[
    U_j=S_j=\varphi_j^{-1}(A),
    \qquad \Psi_j=\restr{\varphi_j}{U_j},
\]
which are open subsets of $\RR^d$, hence subanalytic (relative to themselves, see \cref{sec:prel}) of dimension $d$, together with the real-analytic maps $\varphi_j\colon U_j\to X$; \cref{rem:model-class-analytic-psi} then gives conditions \ref{i} and \ref{ii} of \cref{def:model-class}. Second countability of $\mathcal N$ is a genuine restriction and not a consequence of the ambient structure, since the spaces $X$ arising in the applications, such as $L^\infty(\Omega)$ in \cref{sec:calderon}, are not separable; it is used only to make the atlas countable.
\end{example}

\begin{example}[Sparse parameters]\label[example]{ex:sparse-parameters}
Natural signals are typically structured, and admit sparse representations
\cite{zhang2015survey} in suitable dictionaries such as wavelets \cite{mallat1999wavelet} or
shearlets \cite{kutyniok2012introduction}, and their intrinsic dimension is far smaller than
the ambient one. Compressed sensing \cite{Foucart2013-od} exploits this to recover them
from far fewer measurements, which matters whenever acquisition is slow or expensive.

In mathematical terms, let $e_1,\dots,e_D\in X$ be linearly independent and consider unknown parameters that are $s$-sparse with respect to $(e_j)_{j=1}^D$, namely of the form $a=\sum_{j=1}^D\alpha_je_j$ with $\lvert\operatorname{supp}(\alpha)\rvert\leq s$, where
\[
    \operatorname{supp}(\alpha)=\{j:\alpha_j\neq0\},\qquad\alpha\in\RR^D,
\]
and $\lvert\cdot\rvert$ denotes the cardinality of the set. For $1\leq s\leq D$ we consider the set of $s$-sparse unknowns
\begin{equation}\label{eq:sparse-set}
    \calM_s
    =\bigcup_{\substack{I\subset\{1,\ldots,D\}\\ |I|\leq s}}
      \operatorname{span}\{e_j:j\in I\},
\end{equation}
namely, the union of every subspace generated by at most $s$ of the $e_j$. Note that, for $1\leq s<D$, this set is not a manifold. However, its intersection with $A$ is an analytic model class of dimension at most $s$. Indeed, each support $I$ gives the linear, hence real-analytic, parametrization
\[
    S_I=\{\alpha\in\RR^D:\operatorname{supp}(\alpha)\subset I\}\cap L^{-1}(A),
    \qquad L\alpha=\sum_{j=1}^D\alpha_je_j,
\]
and $S_I$ is the intersection of a linear subspace of dimension $\lvert I\rvert\leq s$ with an open set, hence semianalytic of dimension at most $s$. We refer to \cref{cor:sparse} for the corresponding uniqueness result.   
\end{example}

\subsection{Main results}\label{sec:uniform-recovery}
We now state the main result of the work ensuring uniform identifiability of parameters lying within a $d$-dimensional analytic model class using $2d+1$ scalar measurements arising from independent random samples on $(\calZ,\nu)$.

\begin{theorem}\label{thm:main}
Let $X$ be a real Banach space, $A\subset X$ be open, and $\calM\subset A$ be an analytic model class of dimension $d\geq 1$. Assume one of the two sampling architectures described in \cref{sec:finite-analytic-sampling,sec:gaussian-sampling}, denote by $\mathfrak m$ the corresponding scalar measurement map (see \eqref{eq:analytic-measurement-map} and \eqref{eq:gaussian-measurement-map}) satisfying the corresponding separation and nondegeneracy hypotheses (see \eqref{eq:analytic-separation} and \eqref{eq:gaussian-separation}). Let $M\geq 2d+1$, and let $z_1,\ldots,z_M\in\calZ$ be independent samples with law $\nu$. Then, with probability one, for every $a_1,a_2\in\calM$,
\[
   \text{if } \mathfrak m(a_1,z_i)=\mathfrak m(a_2,z_i),\; \text{for } i=1,\dots,M, \quad \text{ then } a_1=a_2.
\]
\end{theorem}

Equivalently, the set of measurement systems that fail to distinguish at least one pair of distinct parameters has $\nu^M$-measure zero. 

If $\calM$ is an embedded real-analytic manifold and the sampling architecture is the finite-dimensional one of \cref{sec:finite-analytic-sampling}, then \cref{thm:main} is \cite[Theorem 2]{sontag2002experiments}, and \cref{cor:nonuniform} below is \cite[Theorem 1]{sontag1996critical}; see also \cite[\S6.10]{sontag2002experiments}. The two extensions of this work are the Gaussian multilinear architecture, in which the sampling space is infinite-dimensional, and the passage from manifolds to analytic model classes.

The same arguments can be used to show non-uniform recovery, i.e.\ pointwise recovery, using $d+1$ random samples.

\begin{corollary}\label[corollary]{cor:nonuniform}
Under the hypotheses of \cref{thm:main}, fix $a^\dagger\in \calM$. If $M\geq d+1$, then, with probability one, for every $a\in \calM$, 
\[ \text{if } 
    \mathfrak m(a,z_i)=\mathfrak m(a^\dagger,z_i),
    \; \text{for } i=1,\ldots,M, \quad \text{ then } a=a^\dagger.
\]
\end{corollary}

We conclude the section by stating the result concerning uniqueness guarantees with finitely many random measurements for sparse parameters. The sparse example in \cref{ex:sparse-parameters} yields the
following consequence of \cref{thm:main}.

\begin{corollary}\label[corollary]{cor:sparse}
Let $X$ be a real Banach space, $A\subset X$ be open, $e_1,\dots,e_D\in X$ be linearly independent, and let $\calM_s$ be as in \eqref{eq:sparse-set}. Assume one of the two sampling architectures described in \cref{sec:finite-analytic-sampling,sec:gaussian-sampling}, with the corresponding separation hypothesis on $A\cap\calM_s$. In the Gaussian architecture, assume additionally that $F$ is real analytic on $A$. Let $M\geq 2s+1$ and let $z_1,\dots,z_M\in\calZ$ be independent samples with law $\nu$. Then, with probability one, for every $a_1,a_2\in A\cap\calM_s$,
\[ \text{if }
   \mathfrak m(a_1,z_i)=\mathfrak m(a_2,z_i),\; \text{for } i=1,\dots,M, \quad \text{then } a_1=a_2.
\]
\end{corollary}

Notice that the number of measurements is governed by the sparsity level $s$ and not by the ambient dimension $D$, which in the applications is typically much larger. This can be interpreted as an exact, nonlinear analogue of noise-free compressed sensing \cite{Foucart2013-od}. Classical compressed sensing in linear scenarios typically requires $\mathcal O(s\log(D/s))$ measurements for stable recovery \cite{candes2006robust,donoho2006compressed}, and $\min\{2s,D\}$ measurements for identifiability \cite[Theorem 2.13]{Foucart2013-od}. When $2s\leq D$, our bound exceeds the linear identifiability threshold by one measurement, but the setting is fully nonlinear. As everywhere in this paper, the statement is exact and noiseless: unlike quantitative compressed sensing, it does not give robustness to noise or an algorithmic recovery guarantee. Its advantage is that the measurements may be nonlinear and constrained by an inverse problem forward map.

\section{Applications}\label{sec:applications}

\subsection{The Calder\'on problem}\label{sec:calderon}
The Calder\'on problem \cite{calderon2006inverse} is the mathematical foundation of electrical impedance tomography \cite{cheney1999electrical,borcea2002electrical}. It asks whether the electrical conductivity of a body can be uniquely determined by performing current-to-voltage measurements on the boundary.

Let $\Omega\subset\RR^m$, $m\geq2$, be a bounded connected Lipschitz domain. All function spaces in this subsection are real. Set
\[
    L^\infty_+(\Omega)
    =\{\sigma\in L^\infty(\Omega):\essinf_\Omega\sigma>0\}.
\]
For $\sigma\in L^\infty_+(\Omega)$ and $f\in L^2_\diamond(\partial\Omega)=\left\{f\in L^2(\partial\Omega):\int_{\partial\Omega}f\,ds=0\right\}$, let $u_{\sigma,f}\in H^1(\Omega)$ be the weak solution, normalized by $\int_{\partial\Omega}u_{\sigma,f}\,ds=0$, of
\begin{equation*}
  \left\{\begin{aligned}
    &-\nabla\cdot(\sigma\nabla u)=0&&\mathrm{in}~\Omega,\\
    &\sigma\partial_\nu u=f&&\mathrm{on}~\partial\Omega.
    \end{aligned}\right.  
\end{equation*}
The boundary measurements are modeled via the Neumann-to-Dirichlet map
\begin{equation}\label{eq:ND-map}
    \Lambda_\sigma\colon L^2_\diamond(\partial\Omega)\to L^2_\diamond(\partial\Omega),\qquad \Lambda_\sigma f=\restr{u_{\sigma,f}}{\partial\Omega}.
\end{equation}
The Calder\'on problem involves recovering $\sigma$ from $\Lambda_\sigma$.

\begin{remark}\label[remark]{rem:literature-calderon}
It is well established that the problem is uniquely solvable for $m=2$ \cite{nachman1996global, astala2006calderon} (see also \cite{seo1995uniqueness} for a uniqueness result in the planar case with finitely many measurements) or for sufficiently smooth conductivities \cite{sylvester1987global, haberman2013uniqueness,haberman2015uniqueness,caro2016global,chen2026determining}. It has been shown that it is only logarithmically stable \cite{alessandrini1988stable,mandache2001exponential}. Lipschitz type stability estimates have been shown for finite-dimensional unknowns with infinitely many measurements \cite{alessandrini2005lipschitz, rondi2006remark,bacchelli2006lipschitz,beretta2011lipschitz,aspri2022lipschitz} and with finitely many measurements \cite{harrach2019uniqueness,alberti2022infinite,alberti2022inverse,hanke2024lipschitz}. 
\end{remark}
The map $\sigma\mapsto\Lambda_\sigma$ is real-analytic, see \cref{prop:nd-analytic}. Hence, assuming the injectivity of the forward model with full-data, we can apply the abstract recovery result, \cref{thm:main}, in the setting of \cref{sec:gaussian-sampling}, with:
\begin{itemize}
    \item $X=L^\infty(\Omega)$,
    \item $A=L^\infty_+(\Omega)$,
    \item $Z=H=L^2_\diamond(\partial\Omega)$,
    \item $r=2$,
    \item $F(\sigma)(f,g)=\langle\Lambda_{\sigma} f,g\rangle_{L^2(\partial\Omega)}$ for $\sigma\in A$ and $f,g\in H$,
    \item and $\psi=\operatorname{Id}\colon Z\to H$.
\end{itemize}
We obtain the following uniqueness guarantee using finitely many random scalar measurements with boundary samples from $L^2_\diamond(\partial\Omega)$.

\begin{theorem}\label{thm:calderon-gaussian}
Let $\calM\subset L^\infty_+(\Omega)$ be an analytic model class of dimension $d$, and assume that
\[
    \sigma\longmapsto\Lambda_\sigma
\]
is injective on $\calM$. Let $\gamma$ be a nondegenerate Gaussian measure on $L^2_\diamond(\partial\Omega)$. Let $M\geq2d+1$ and let $(f_1,g_1),\dots,(f_M,g_M)\in L^2_\diamond(\partial\Omega)\times L^2_\diamond(\partial\Omega)$ be independent samples with law $\gamma\otimes\gamma$. Then, with probability one, for every $\sigma_1,\sigma_2\in\calM$,
\[ \text{if }
   \langle\Lambda_{\sigma_1} f_i,g_i\rangle_{L^2(\partial\Omega)}=\langle\Lambda_{\sigma_2} f_i,g_i\rangle_{L^2(\partial\Omega)},\; \text{for } i=1,\dots,M,\quad \text{then } \sigma_1=\sigma_2.
\]
\end{theorem}

\begin{remark}\label[remark]{rem:calderon-injectivity}
The only problem-specific hypothesis of \cref{thm:calderon-gaussian} is
injectivity of $\sigma\mapsto\Lambda_\sigma$ on the chosen model class, see \cref{rem:literature-calderon}. For $m=2$ and
$\Omega$ bounded, simply connected and Lipschitz, the Dirichlet-to-Neumann map
$\Lambda_\sigma^{\mathrm{DN}}$ determines $\sigma\in L^\infty_+(\Omega)$
\cite{astala2006calderon}. Since the Neumann-to-Dirichlet and
Dirichlet-to-Neumann maps are inverses of each other between
$H^{-1/2}_\diamond(\partial\Omega)$ and $H^{1/2}_\diamond(\partial\Omega)$, and
$L^2_\diamond(\partial\Omega)$ is dense in $H^{-1/2}_\diamond(\partial\Omega)$,
the map $\sigma\mapsto\Lambda_\sigma$ is injective on $L^\infty_+(\Omega)$ as
well. Hence \cref{thm:calderon-gaussian} holds, with no additional hypotheses,
for every analytic model class $\calM$. For $m\geq3$, the same argument shows
that the injectivity hypothesis is automatic when $\calM$ consists of $W^{1,m}$ conductivities \cite{haberman2015uniqueness,chen2026determining}, or of piecewise constant, piecewise linear
or piecewise analytic conductivities on a fixed partition, under the respective
assumptions of
\cite{alessandrini2005lipschitz,alessandrini2017lipschitz,harrach2019uniqueness}.
\end{remark}

\subsubsection{Finite-dimensional boundary sampling}\label{sec:calderon-finite}
In the previous section we established a recovery guarantee for the Calder\'on problem where the input data are sampled from the infinite-dimensional space $L^2_\diamond(\partial\Omega)$.
However, since the parameter space is finite-dimensional, it is possible to appropriately truncate an orthonormal basis $(e_j)_{j\in\NN}$ for $L^2_\diamond(\partial\Omega)$ and consider the finite-dimensional space
\begin{equation}\label{eq HN}
    H_N=\operatorname{span}\{e_1,\dots,e_N\}\subset\LDd,\quad\text{for some}~N\in\NN,
\end{equation}
as the sampling space.

The following theorem relies on a uniform truncation statement, see \cref{prop:truncation-preserves-injectivity}. The truncation must preserve not only injectivity on the compact search set, but also injectivity of the differential of the truncated forward map. For this reason, and unlike \cref{thm:calderon-gaussian}, this is the one result of the paper that is stated for an embedded real-analytic manifold rather than for a general analytic model class: the argument is a compactness argument on normalized secants and unit tangent vectors, and it uses the tangent spaces of $\calM$. In the following, we denote by $\calK\bigl(L^2_\diamond(\partial\Omega)\bigr)$ the set of compact operators acting between $L^2_\diamond(\partial\Omega)$ and itself.

\begin{theorem}\label{thm:calderon-finite}
Let $\calM\subset L^\infty(\Omega)$ be a second-countable embedded real-analytic manifold of dimension $d$, and let
\[
    K\subset L^\infty_+(\Omega)\cap\calM
\]
be compact. Assume that:
\begin{enumerate}[label=\textup{(\roman*)},leftmargin=*]
    \item $\sigma\mapsto\Lambda_\sigma$ is injective on $K$;
    \item for every $\sigma\in K$, the differential
    \[
        d\Lambda_\sigma\colon T_\sigma\calM\longrightarrow
        \calK\bigl(L^2_\diamond(\partial\Omega)\bigr)
    \]
    is injective.
\end{enumerate}
Then there exists $N_0\in\NN$ such that, for every $N\geq N_0$, the following holds. Let $\rho_N$ be a probability measure on $H_N$ that is absolutely continuous with respect to the Lebesgue measure. Let $M\geq2d+1$ and let $(f_1,g_1),\dots,(f_M,g_M)\in H_N\times H_N$ be independent samples with law $\rho_N\otimes\rho_N$. Then, with probability one, for every $\sigma_1,\sigma_2\in K$,
\[ \text{if }
   \langle\Lambda_{\sigma_1} f_i,g_i\rangle_{L^2(\partial\Omega)}=\langle\Lambda_{\sigma_2} f_i,g_i\rangle_{L^2(\partial\Omega)},\; \text{for } i=1,\dots,M,\quad \text{then } \sigma_1=\sigma_2.
\]
\end{theorem}

\begin{remark}\label[remark]{rem:compact-calderon-set}
A common choice is
\[
    K=\calM\cap\left\{\sigma\in L^\infty(\Omega):\lambda^{-1}\leq\sigma\leq\lambda\ \text{a.e.\ in}~\Omega\right\}
\]
for some $\lambda>1$, provided this intersection is compact in $\calM$. For instance, this is the case whenever $\calM\subset L^\infty(\Omega)$ is a finite-dimensional  linear subspace. However, compactness is an assumption; it does not follow merely from boundedness in the ambient $L^\infty$-norm for an arbitrary embedded manifold.
\end{remark}

\begin{remark}\label[remark]{rem:calderon-hypotheses}
Implementing this finite-dimensional truncation relies on the injectivity of both the forward operator, discussed earlier, and the linearized operator. The injectivity of the latter follows from the completeness of products of suitable solutions across several smooth settings \cite{sylvester1987global,paivarinta2003complex}; for structured conductivity classes, we refer to \cite{lechleiter2008newton,harrach2010exact}. The hypotheses of \cref{thm:calderon-finite} must be explicitly verified for the specific tangent spaces of the underlying model.
\end{remark}

\subsubsection{Conductivities on moving polygonal partitions}
\label{sec:calderon-moving}
 
In electrical impedance tomography, the unknown geometry may be as important as
the conductivity values. Unlike the fixed partitions mentioned in
\cref{rem:calderon-injectivity}, the following construction allows both the
interfaces and the values to vary. Its coefficient parametrization is generally
discontinuous in $L^\infty(\Omega)$, but its boundary response depends
analytically on the parameters. The same setting, with piecewise constant conductivities on polygonal partitions whose vertices and values are both unknown, is considered in \cite{beretta2017differentiability,beretta2018reconstruction}. In \cite{beretta2017differentiability}, the authors prove that the Dirichlet-to-Neumann and Neumann-to-Dirichlet maps are differentiable with respect to movements of the vertices of polygonal inclusions; in \cite{beretta2018reconstruction}, shape derivatives with respect to such movements are used to reconstruct both the partition and the conductivity values from boundary measurements. In three dimensions, Lipschitz stable determination of polyhedral conductivity inclusions from local boundary measurements is proved in \cite{aspri2022lipschitz}.
 
Let $\Omega\subset\RR^2$ be a bounded, simply connected polygonal domain, and fix
a reference triangulation $\mathcal T=\{P_1,\dots,P_L\}$ of $\overline\Omega$:
the $P_l$ are open triangles, $\overline\Omega=\bigcup_{l}\overline{P_l}$, and two
distinct closed triangles are either disjoint or share exactly one vertex or
exactly one edge. Let $p_1,\dots,p_N\in\Omega$ be the interior vertices of
$\mathcal T$; all other vertices lie on $\partial\Omega$ and are called boundary
vertices. We keep the connectivity of $\mathcal T$ (which vertices span which
triangle) and the boundary vertices fixed, and we move the interior vertices: for
$v=(v_1,\dots,v_N)\in\RR^{2N}$, the point $v_j$ is the new position of
$p_j$.
 
Given $v$, let $\Phi_v\colon\overline{\Omega}\to\RR^2$ be the continuous map that
is affine on each $\overline{P_l}$ and satisfies $\Phi_v(p_j)=v_j$ for
$j=1,\dots,N$ and $\Phi_v(z)=z$ for every boundary vertex $z$. Explicitly, if
$a_0,a_1,a_2$ are the vertices of $P_l$, labeled counterclockwise, and
$b_0,b_1,b_2$ are their new positions, then
\begin{equation}\label{eq:Phi-v-explicit}
    \Phi_v(x)=b_0+Q_l(v)\,(x-a_0)\quad\text{for }x\in\overline{P_l},
    \qquad
    Q_l(v)=\bigl[\,b_1-b_0\ \ b_2-b_0\,\bigr]\,
           \bigl[\,a_1-a_0\ \ a_2-a_0\,\bigr]^{-1},
\end{equation}
where $[\,\xi\ \ \eta\,]$ denotes the $2\times2$ matrix with columns
$\xi,\eta\in\RR^2$; thus $Q_l(v)=D(\Phi_v|_{P_l})$. The affine maps of two
adjacent triangles agree on their common edge, because both are affine along it
and coincide at its endpoints; hence $\Phi_v$ is well defined and continuous.
Every edge contained in $\partial\Omega$ has fixed endpoints, so
$\Phi_v=\mathrm{Id}$ on $\partial\Omega$.
 
By \eqref{eq:Phi-v-explicit}, $Q_l(v)$ depends affinely on $v$, and
\[
    \det Q_l(v)
    =\frac{\det\bigl[\,b_1-b_0\ \ b_2-b_0\,\bigr]}
          {\det\bigl[\,a_1-a_0\ \ a_2-a_0\,\bigr]}
\]
is a polynomial of degree at most two in $v$. The denominator is twice the area
of $P_l$, and the numerator is twice the signed area of the moved triangle with
vertices $b_0,b_1,b_2$. Hence $\det Q_l(v)>0$ exactly when the moved triangle is
nondegenerate and still labeled counterclockwise; in words, $\Phi_v$ neither
flattens nor flips $P_l$. 

In particular, for $v\in\RR^{2N}$ such that $\det Q_l(v)>0$, we have that $\Phi_v$ maps $\overline\Omega$
bijectively onto itself. Indeed, in $\mathcal T$, two triangles sharing
an edge lie on opposite sides of it. Since $\Phi_v$ flips no triangle, the two
moved triangles still lie on opposite sides of their moved common edge. Similarly, a triangle with an edge on $\partial\Omega$ stays on the inner side of that edge, which does not move. Thus the moved triangles do not overlap, leave no gaps in $\Omega$, and do not
stick out of $\overline\Omega$. A rigorous proof, which also covers points on
edges and at vertices, is given in \cite[Theorem~1]{lipman2014bijective}. It follows that the moved triangles $\Phi_v(P_l)$ form a new triangulation of $\overline\Omega$ and $v_j=\Phi_v(p_j)\in\Omega$ necessarily.

Set
\[
    U=\bigl\{(v,c)\in\RR^{2N}\times\RR^L:
        \det Q_l(v)>0,\ c_l>0,\ l=1,\dots,L\bigr\}.
\]
Observe that the $2L$ functions $(v,c)\mapsto\det Q_l(v)$ and $(v,c)\mapsto c_l$ are
polynomials, so $U$ is open and semialgebraic. It is nonempty: the reference
positions $v^0=(p_1,\dots,p_N)$ give $\Phi_{v^0}=\mathrm{Id}$ and $Q_l(v^0)=I$,
so $(v^0,c)\in U$ whenever all $c_l>0$.  

We consider the following class of positive piecewise constant conductivities on moving triangolations. Let
\[
    \sigma_t=\widehat\sigma_c\circ\Phi_v^{-1}
            =\sum_{l=1}^L c_l\chi_{\Phi_v(P_l)},\qquad t=(v,c)\in U,
\]
where $\widehat\sigma_c=\sum_{l=1}^L c_l\chi_{P_l}$.

Polygonal cells can be represented by grouping triangles and imposing equal conductivity values within each group. In particular, an interior group with a common value, surrounded by triangles sharing another value, describes a polygonal inclusion. 

\cref{thm:main} gives the following
uniqueness guarantee.
\begin{corollary}\label[corollary]{cor:moving}
Let $S\subset U$ be subanalytic of dimension at most $d$, and set
$\calM$ as in \eqref{eq:an-mod-moving}. Let $\gamma$ be a nondegenerate Gaussian
measure on $L^2_\diamond(\partial\Omega)$, and let
$(f_1,g_1),\dots,(f_M,g_M)$ be independent samples with law
$\gamma\otimes\gamma$, where $M\geq 2d+1$.
Then the set 
\begin{equation}\label{eq:an-mod-moving}
    \calM=\{\sigma_t:t\in S\}
\end{equation}
is an analytic model class of dimension $d$, and, almost surely, for every $t_1,t_2\in S$,
\[ \text{if }
   \langle\Lambda_{\sigma_{t_1}} f_i,g_i\rangle_{L^2(\partial\Omega)}
   =\langle\Lambda_{\sigma_{t_2}} f_i,g_i\rangle_{L^2(\partial\Omega)},\;
   \text{for } i=1,\dots,M,\quad \text{then } \sigma_{t_1}=\sigma_{t_2}.
\]
\end{corollary}

\begin{remark}\label[remark]{rem:moving}
Moving an interface across which the conductivity jumps generally changes
$\sigma_t$ by a fixed positive amount in $L^\infty(\Omega)$, however small the
displacement. Thus the analytic response parametrization above does not satisfy the requirements of \cref{rem:model-class-analytic-psi}.
The corollary identifies conductivities, not necessarily their parameters:
interfaces between equal values are invisible, and the mesh representation may
be redundant.

Taking $S=U$ gives the dimension bound $d=2N+L$. Fixing vertex coordinates or
imposing equalities among cell values can reduce this bound.
\end{remark}

\subsection{Inverse medium scattering}\label{sec:scattering}
The inverse medium scattering problem \cite{colton1998inverse, kirsch2011introduction} involves recovering an unknown complex-valued refractive index of a medium from far-field measurements.

The mathematical model is the following, see \cite[Chapter 8]{colton1998inverse} for further details. Let $k>0$ be the fixed wavenumber and $n \in L^\infty(\RR^3;\CC)$ be a complex-valued refractive index. We assume the medium of interest is supported within a known bounded ball $B\subset\RR^3$, namely,  $n=1$ in $\RR^3\setminus B$. Let the open set of admissible parameters be
\begin{equation*}
    \mathcal A_k
    =\{n\in L^\infty(B;\CC):I-V_n\text{ is invertible on }L^2(B)\},
\end{equation*}
where $V_n\colon L^2(B)\to L^2(B)$ is the volume potential operator \cite[Theorem 8.2]{colton1998inverse} defined by
\begin{equation}\label{eq:volume-potential}
    V_n v
    =k^2\int_B G_k(\cdot,y)(n(y)-1)v(y)\,dy,
\end{equation}
and
\[
    G_k(x,y)=\frac{e^{ik|x-y|}}{4\pi|x-y|}
\]
is the outgoing fundamental solution of $\Delta+k^2$. For $n\in\mathcal A_k$ and for an incident plane wave $u^{\rm i}_\theta(x)=e^{ik x\cdot \theta}$ that solves
\begin{equation*}
    \Delta u_\theta^{\rm i} +k^2 u_\theta^{\rm i}=0\quad\text{in}~\RR^3,
\end{equation*} 
we denote by $u_{n,\theta}$ the corresponding total field fulfilling 
\begin{equation}\label{eq phisical model scattering}
    \left\{\begin{aligned}
    &\Delta u_{n,\theta}+k^2nu_{n,\theta}=0&&\mathrm{in}~\RR^3,\\
    &u_{n,\theta}=u_\theta^{\rm i}+u_{n,\theta}^s&&\mathrm{in}~\RR^3,
    \end{aligned}\right.   
\end{equation}
where $u_{n,\theta}^s$ is the scattered field satisfying the Sommerfeld radiation condition
\begin{equation*}
    \lim_{r\to+\infty}r\bigg(\partial_r u^s_{n,\theta}-iku^s_{n,\theta}\bigg)=0,
\end{equation*}
where $\partial_r$ denotes the radial derivative. Notice that,
by the Lippmann--Schwinger equation \cite[Equation 8.13]{colton1998inverse}, we have
\begin{equation}\label{eq:lippmann-schwinger}
    u^s_{n,\theta}=k^2\int_{B}G_k(\cdot,y)(n(y)-1)u_{n,\theta}(y)\,dy=V_n u_{n,\theta},
\end{equation}
and so, by the second equation in \eqref{eq phisical model scattering}, for $n\in\mathcal{A}_k$ the total field is well-defined and uniquely given by
\[
    u_{n,\theta}=(I-V_n)^{-1}u_\theta^{\rm i}.
\]
The corresponding far-field pattern on the sphere $S^2$ is given by \begin{equation}\label{eq:far-field}
    u_n^\infty(\widehat x,\theta)
    =\frac{k^2}{4\pi}\int_B e^{-ik\widehat x\cdot y}(n(y)-1)u_{n,\theta}(y)\,dy,
    \qquad (\widehat x,\theta)\in S^2\times S^2.
\end{equation}
The inverse medium scattering problem involves recovering $n$ in $B$ from the far-field pattern $u_n^\infty$ on $S^2\times S^2$.

\begin{remark}\label[remark]{rem:literature-scattering}
This inverse problem is known to be uniquely solvable when we have access to the full far-field pattern on $S^2\times S^2$ \cite{novikov1988multidimensional, nachman1988reconstructions,ramm1988recovery,colton1998inverse,kirsch2011introduction}. It has been shown also that, under certain assumptions, probing the medium with a single planar wave and measuring the corresponding far-field pattern on $S^2$ uniquely determines the values of a perturbation to the refractive
index on the corners of its support \cite{blaasten2020recovering}, and that a polyhedral inclusion can be reconstructed with a single far-field pattern measurement \cite{blaasten2021corners}. We refer also to \cite{di2026discretization} for the inverse problem of reconstructing inhomogeneities by performing a finite number of scattering measurements of acoustic type in the time-harmonic setting. The infinite-dimensional problem has been shown to be only logarithmically stable \cite{stefanov1990stability}. We refer to \cite{bourgeois2013remark} for Lipschitz stability estimates when the unknown is finite-dimensional and the full far-field pattern is available, and to \cite{alberti2022infinite} for a Lipschitz stability estimate for finite-dimensional unknowns with only finitely many measurements. 
\end{remark}

We can apply  \cref{thm:main}, in the setting of \Cref{sec:finite-analytic-sampling}, with
\begin{itemize}
    \item $X=L^\infty(B;\mathbb{C})$,
    \item $A=\mathcal A_k$,
    \item $Z=S^2\times S^2$,
    \item ${\rm vol}_Z=m_{S^2}\otimes m_{S^2}$, where $m_{S^2}$ is the surface measure of $S^2$,
    \item and $\mathfrak m\colon \mathcal{A}_k\times (S^2\times S^2)\to \mathbb K$ given by $\mathfrak m(n,(\widehat x,\theta))=u^\infty_n(\widehat x,\theta)$.
\end{itemize}
This leads to a recovery guarantee for finite-dimensional refractive indices using only finitely many incident directions and for each of them measuring the corresponding far-field pattern only at a single location. 

The assumption we need is the injectivity of the forward model on the chosen model class with full-data, discussed in \cref{rem:scattering-injectivity}. As for
admissibility, we notice that the Helmholtz volume potential $V_n$
\eqref{eq:volume-potential} maps $L^2(B)$ continuously into $H^2(B)$
\cite[Theorem 8.2]{colton1998inverse}, hence $V_n\colon L^2(B)\to L^2(B)$ is
compact by the Rellich theorem. Since $V_n$ is compact, injectivity of $I-V_n$
implies invertibility by the Fredholm alternative. In particular, every standard
admissible class for which the homogeneous radiating problem is uniquely
solvable is contained in $\mathcal A_k$; classical examples are real-valued
indices and complex-valued indices with nonnegative imaginary part, see \cite{colton1998inverse,kirsch2011introduction}.

\begin{theorem}\label{thm:scattering}
Let $\calM\subset\mathcal A_k$ be an analytic model class of dimension $d$. Assume that
\[
    n\longmapsto u_{n}^\infty\in L^2(S^2\times S^2)
\]
is injective on $\calM$. Let $\mu$ be a probability measure on $S^2\times S^2$ satisfying \eqref{eq:mu-ac}. Let $M\geq2d+1$ and let $(\widehat x_1,\theta_1),\dots,(\widehat x_M,\theta_M)\in S^2\times S^2$ be independent samples with law $\mu$. Then, with probability one, for every $n_1,n_2\in\calM$,
\[ \text{if }
    u_{n_1}^\infty(\widehat x_i,\theta_i)=u_{n_2}^\infty(\widehat x_i,\theta_i), \; \text{for } i=1,\dots,M, \quad \text{then } n_1=n_2.
\]
\end{theorem}

\begin{remark}\label[remark]{rem:scattering-injectivity}
As in \cref{rem:calderon-injectivity}, the only problem-specific hypothesis is
full-data injectivity on $\calM$, which is inherited by subsets. Both this and
admissibility hold automatically whenever $n$ has nonnegative imaginary part. Indeed, the outgoing
Lippmann--Schwinger equation is then uniquely solvable, so $n\in\mathcal A_k$
\cite[Section~3.4]{novikov2025holographic}, and the far-field pattern at a fixed
frequency on $S^2\times S^2$ uniquely determines $n$, see
\cite{novikov1988multidimensional,nachman1988reconstructions,ramm1988recovery}
for real $n$, and \cite{novikov1988multidimensional,novikov2025holographic} for
complex $n$; see also \cite[Chapter~10]{colton1998inverse} and
\cite{kirsch2011introduction}. 
\end{remark}

\section{Proof of the abstract results}\label{sec:proof-main}

\subsection{Preliminaries}
By \cref{def:model-class}, the measurements depend analytically on the parameters, which is what allows us to treat the zero sets below as semianalytic subsets of finite-dimensional real-analytic manifolds, in the sense of \cref{sec:prel}. The strategy of the proof of \cref{thm:main} is Sontag's \cite{sontag1996critical,sontag2002experiments}, where it is carried out for the finite-dimensional sampling architecture and embedded manifolds; \cref{lem:fiber-dimension} is \cite[Proposition A.2]{sontag1996critical}, to which we refer for the dimension bound.

\begin{lemma}\label[lemma]{lem:analytic-zero-set}
Let $\mathcal N$ be a connected real-analytic manifold of dimension $q$, and let $f\colon\mathcal N\to\KK$ be real-analytic and not identically zero. Then $f^{-1}(0)$ is semianalytic and
\[
    \dim f^{-1}(0)\leq q-1.
\]
\end{lemma}

\begin{proof}
For $\KK=\CC$, replace $f$ by the real-valued analytic function $|f|^2$. Assume now $\KK=\RR$. The zero level set is semianalytic by definition. In particular, it admits a stratification into analytic submanifolds \cite[Proposition 2.10]{bierstone1988semianalytic}. If the zero set had dimension $q$, then the analytic stratification would contain a $q$-dimensional stratum, and hence the zero set would contain a nonempty open subset of $\mathcal N$. The real-analytic identity principle would then imply $f\equiv0$ on the connected manifold $\mathcal N$, a contradiction.
\end{proof}

\begin{lemma}\label[lemma]{lem:fiber-dimension}
Let $\calX$ and $\calY$ be finite-dimensional real-analytic manifolds. Let $P\subset\calX$ be a relatively compact subanalytic set of dimension $p$, and let
\[
    \widetilde E\subset P\times\calY
\]
be subanalytic. For $x\in P$, write
\[
    E_x=\{y\in\calY:(x,y)\in\widetilde E\}.
\]
If $\dim E_x\leq k$ for every $x\in P$, then
\[
    \dim\widetilde E\leq p+k.
\]
Moreover, denoting by $\pi_\calY\colon\calX\times\calY\to\calY$ the canonical projection, we have that $E=\pi_\calY(\widetilde E)$ is subanalytic and satisfies
\[
    \dim E\leq p+k.
\]
\end{lemma}
\begin{proof}
The bound $\dim\widetilde E\leq p+k$ is \cite[Proposition A.2, part 2]{sontag1996critical},
applied to the projection $\pi_\calX\colon\calX\times\calY\to\calX$.

For the second assertion, let $y\in\calY$ and let $V$ be a relatively compact subanalytic coordinate neighbourhood of $y$, which exists by local compactness of $\calY$. The set
$\widetilde E\cap(P\times V)$ is subanalytic and relatively compact, because $P$ and $V$
are; hence $E\cap V=\pi_\calY\bigl(\widetilde E\cap(P\times V)\bigr)$ is subanalytic, by
the proper image theorem \cite[\S3]{bierstone1988semianalytic}. Since subanalyticity is a
local property, $E$ is subanalytic. Finally, the image of a subanalytic set under a
projection has dimension at most that of the set \cite{bierstone1988semianalytic}, so
\[
    \dim E\ \leq\ \dim\widetilde E\ \leq\ p+k. \qedhere
\]
\end{proof}

\begin{lemma}\label[lemma]{lem:compact-exhaustion}
Let $U\subset\RR^n$ be open, let $S\subset U$ be subanalytic of dimension at most $d$, and let $O\subset U$ be open. Then there are countably many compact subanalytic sets $K_l\subset O$, each of dimension at most $d$, such that
\[
    S\cap O\subset\bigcup_{l\geq1}K_l .
\]
\end{lemma}

\begin{proof}
If $O=\varnothing$, take $K_l=\varnothing$ for every $l$.
Otherwise, let $C=\overline{S}^{\,U}$ denote the closure of $S$
relative to $U$. Then $C$ is subanalytic in $U$ and
$\dim C=\dim S$; see \cite[Section~3]{bierstone1988semianalytic}.

Choose countably many open balls $B_l$ such that
$\overline{B_l}\subset O$ and $\bigcup_{l\geq1}B_l=O$, and set
\[
    K_l=C\cap\overline{B_l}.
\]
Since $C$ is closed relative to $U$ and
$\overline{B_l}$ is a compact subset of $U$, the set $K_l$
is compact and contained in $O$.
It is subanalytic in $U$, being the intersection of two
subanalytic sets. Since it is compact and contained in $U$,
it is also subanalytic in $\RR^n$.
Moreover,
\[
    \dim K_l\leq\dim C=\dim S\leq d.
\]
Finally, every $x\in S\cap O$ belongs to some $B_l$ and hence
to $K_l$, proving the required covering.
\end{proof}

\subsection{A uniform finite-dimensional truncation}
The following lemma is used only for Gaussian multilinear sampling (\Cref{sec:gaussian-sampling}).

\begin{lemma}\label[lemma]{lem:multilinear-truncation}
Let $Z$ be a separable Hilbert space with orthonormal basis $(e_j)_{j\geq1}$, and set $Z_N=\operatorname{span}\{e_1,\ldots,e_N\}$. Let $P$ be a compact topological space, and let
\[
    \Phi\colon P\times Z^r\longrightarrow\KK
\]
be such that $\Phi(x,\cdot)\in\mathcal{L}^r(Z;\KK)\setminus\{0\}$ for every $x\in P$, and $\Phi(\cdot,v)$ is continuous for every $v\in Z^r$. Then there exists $N$ such that
\[
    \restr{\Phi(x,\cdot)}{Z_N^r}\not\equiv0
    \qquad\text{for every }x\in P.
\]
\end{lemma}

\begin{proof}
Fix $x\in P$. We claim that 
\[
\Phi\bigl(x,e_{j_1(x)},\ldots,e_{j_r(x)}\bigr)\neq0
\]
for some $j_1(x),\dots,j_r(x)\ge 1$. Indeed, otherwise, we would have
\[
\Phi\bigl(x,e_{j_1},\ldots,e_{j_r}\bigr)=0,\qquad \forall j_1,\dots,j_r\ge 1,
\]
which, by multi-linearity and continuity of $\Phi(x,\cdot)$ would imply that $\Phi(x,\cdot)\equiv 0$, a contradiction.

 By continuity of $\Phi$ with respect to the first variable, the same inequality holds in a neighborhood of $x$. A finite subcover of $P$ yields finitely many index tuples; taking $N$ larger than all indices proves the claim.
\end{proof}

\subsection{Proofs of the abstract results}

\begin{proof}[Proof of \cref{thm:main}]
Let $(U_j,S_j,\Psi_j)_{j\geq1}$ be an analytic parametrization of $\calM$, as in
\cref{def:model-class}, and for every pair of indices $(j,k)$ set
\[
    U_{jk}=U_j\times U_k,
    \qquad
    \mathcal P_{jk}
    =\bigl\{(t_1,t_2)\in S_j\times S_k:\Psi_j(t_1)\neq\Psi_k(t_2)\bigr\} .
\]
Thus $U_{jk}$ is an open subset of $\RR^{n_j+n_k}$ and $\mathcal P_{jk}\subset U_{jk}$ is
contained in the subanalytic set $S_j\times S_k$, of dimension at most
$\dim S_j+\dim S_k\leq2d$. By \eqref{eq:model-class}, every pair of distinct elements of
$\calM$ is of the form $\bigl(\Psi_j(t_1),\Psi_k(t_2)\bigr)$ with
$(t_1,t_2)\in\mathcal P_{jk}$, for some pair of indices $(j,k)$. Since there are countably
many pairs $(j,k)$, by subadditivity of the measure $\nu$ it suffices to show that, for each
$(j,k)$, the set of measurement systems failing to distinguish a pair in
$\mathcal P_{jk}$ has measure zero.

Fix one such pair of indices $(j,k)$. Define
\[
    \Phi\bigl((t_1,t_2),z\bigr)
    =\mathfrak m_j(t_1,z)-\mathfrak m_k(t_2,z),
    \qquad (t_1,t_2)\in U_{jk},\ z\in\calZ,
\]
which is real analytic by \cref{def:model-class}. By the relevant separation assumption
(see \eqref{eq:analytic-separation} and \eqref{eq:gaussian-separation}), $\Phi(p,\cdot)$
is not identically zero for every $p=(t_1,t_2)\in\mathcal P_{jk}$, because $\Psi_j(t_1)$
and $\Psi_k(t_2)$ are two distinct elements of $\calM$. Hence $\mathcal P_{jk}$ is
contained in
\[
    O_{jk}=\bigl\{p\in U_{jk}:\Phi(p,\cdot)\not\equiv0\ \text{on }\calZ\bigr\},
\]
which is open, being the union over $z\in\calZ$ of the open sets where $\Phi(p,z)\neq0$.
By \cref{lem:compact-exhaustion}, there are compact subanalytic sets $K_l\subset O_{jk}$
of dimension at most $2d$ with $\mathcal P_{jk}\subset\bigcup_{l\geq1}K_l$, and again by
subadditivity it suffices to treat each of them separately.

Fix one such compact set $K=K_l$; recall that $\Phi(p,\cdot)$ is not identically zero for
every $p\in K$, since $K\subset O_{jk}$. The set of measurement systems failing to
distinguish a pair in $K$ is given by
\begin{equation*}
    E_K=\{(z_1,\dots,z_M)\in\calZ^M: \Phi(p,z_i)=0,\,i=1,\dots,M,~\text{for some}~p\in K\}.
\end{equation*}
By compactness of $K$ and continuity of $\Phi$, the set $E_K$ is closed, and in
particular measurable.

\smallskip
\noindent\emph{Finite-dimensional analytic sampling (\Cref{sec:finite-analytic-sampling}).} This is the argument of \cite[Theorem 2]{sontag2002experiments}, of
which the present case is the extension to analytic model classes.
Let $q=\dim Z$. For $p\in K$, \cref{lem:analytic-zero-set} gives
$\dim\Phi(p,\cdot)^{-1}(0)\leq q-1$. For $M$ measurements, consider the subanalytic
incidence set
\[
    \widetilde E_K
    =\left\{(p,z_1,\ldots,z_M)\in K\times Z^M:
      \Phi(p,z_i)=0\ \text{for every }i\right\},
\]
of which $E_K$ is the projection onto $Z^M$. Its fibre over $p$,
$\prod_{i=1}^M\Phi(p,\cdot)^{-1}(0)$, has dimension at most $M(q-1)$, so that
\cref{lem:fiber-dimension} yields
\[
    \dim E_K\leq \dim\widetilde E_K
    \leq 2d+M(q-1)\leq qM-1,
\]
the last inequality because $M\geq2d+1$. Thus $E_K$ is a subanalytic subset of $Z^M$
with positive codimension, hence $\mu^M(E_K)=0$ because
\eqref{eq:mu-ac}. 

\smallskip
\noindent\emph{Gaussian multilinear sampling (\Cref{sec:gaussian-sampling}).}
Let $Q$ be the covariance operator of $\mu$. Since $Q$ is positive, trace class, and injective there is an orthonormal eigenbasis $(e_i)_{i\geq1}$ of $Z$ with strictly positive eigenvalues. Apply \cref{lem:multilinear-truncation} to $\Phi$ on $K$ and obtain $N$ such that the restriction of $\Phi(p,\cdot)$ to $Z_N^r$ is nonzero for every $p\in K$.
Let $P_N$ be the orthogonal projection of $Z$ onto $Z_N$ and set
$\mu_N=(P_N)_\#\mu$ and $\mu_N^\perp=(I-P_N)_\#\mu$. Because $Z_N$ is spanned
by eigenvectors of $Q$, the Gaussian variables $\langle z,u\rangle$ and
$\langle z,v\rangle$ are uncorrelated, hence independent, for $u\in Z_N$ and
$v\in Z_N^\perp$. Consequently, under the identification $Z\cong Z_N\times Z_N^\perp$,
\begin{equation}\label{eq:gaussian-splitting}
    \mu=\mu_N\otimes\mu_N^\perp ,
\end{equation}
and $\mu_N$ is a nondegenerate Gaussian measure on the $N$-dimensional space $Z_N$,
hence has a smooth positive density with respect to the Lebesgue measure
\cite{Bogachev1998-oc}.

Fix tails $w_i\in(Z_N^\perp)^r$, $i=1,\ldots,M$. By \cref{def:model-class} and
\eqref{eq:gaussian-measurement-map},
\[
    \Phi(p,z)=T_p\bigl(\psi z^1,\ldots,\psi z^r\bigr),
    \qquad T_p=F_j(t_1)-F_k(t_2)\in\calL^r(H;\KK),
\]
so that $\Phi(p,\cdot)$ is $\RR$-multilinear on $Z^r$. Expanding each argument of
$\Phi(p,z+w_i)$ by multilinearity gives
\[
    \Phi(p,z+w_i)=\Phi(p,z)+R_i(z),\qquad z\in Z_N^r,
\]
where every term of $R_i$ carries at least one argument taken from $w_i$, and is
therefore a polynomial in $z$ of degree strictly smaller than $r$. By the choice of $N$,
the restriction of $\Phi(p,\cdot)$ to $Z_N^r$ is a nonzero homogeneous polynomial of
degree $r$; since homogeneous polynomials of different degrees are linearly independent,
$z\mapsto\Phi(p,z+w_i)$ is a nonzero polynomial on $Z_N^r$, and
\cref{lem:analytic-zero-set} bounds the dimension of its zero set by $Nr-1$.

The sliced incidence set
\[
    \widetilde E_K(w_1,\ldots,w_M)
    =\left\{
      \begin{aligned}
        &(p,z_1,\ldots,z_M)\in K\times(Z_N^r)^M:\\[-2pt]
        &\Phi(p,z_i+w_i)=0,~i=1,\dots,M
      \end{aligned}
    \right\}
\]
is subanalytic and has dimension at most $2d+M(Nr-1)\leq MNr-1$, where we used
$M\geq2d+1$. By \cref{lem:fiber-dimension}, its projection $E_K(w_1,\ldots,w_M)$ onto
$(Z_N^r)^M$ is subanalytic of dimension at most $MNr-1$, hence Lebesgue-negligible, and
since $(\mu_N^r)^M$ has a smooth density with respect to Lebesgue measure,
\[
    (\mu_N^r)^M\bigl(E_K(w_1,\ldots,w_M)\bigr)=0.
\]
Integrating in the tails and using \eqref{eq:gaussian-splitting} and Fubini's theorem,
\begin{equation*}
       (\mu^r)^M(E_K)
   =\int_{(Z_N^\perp)^{rM}}(\mu_N^r)^M\bigl(E_K(w_1,\ldots,w_M)\bigr)
     \,d(\mu_N^\perp)^{rM}(w_1,\dots,w_M)=0 . \qedhere
\end{equation*}
\end{proof}

\begin{proof}[Proof of \cref{cor:nonuniform}]
Repeat the proof with the one-parameter sets
\[
    \mathcal P_j=\bigl\{t\in S_j:\Psi_j(t)\neq a^\dagger\bigr\}\subset U_j,
    \qquad j\geq1,
\]
in place of the sets of pairs $\mathcal P_{jk}$, and with the responses
\[
    \Phi(t,z)=\mathfrak m_j(t,z)-\mathfrak m(a^\dagger,z),
\]
which are real analytic on $U_j\times\calZ$ because $a^\dagger$ is fixed. The dimension of $\mathcal P_j$ is at most $d$, so the same computations give that $E_K$ has measure zero as soon as $M\ge d+1$.
\end{proof}

\section{Proofs of the results on the Calder\'on  problem}\label{sec:proof-calderon}

The following result is standard.
\begin{proposition}\label[proposition]{prop:nd-analytic}
The map
\[
    \Lambda\colon L^\infty_+(\Omega)\longrightarrow
    \calK\bigl(L^2_\diamond(\partial\Omega)\bigr),
    \qquad \sigma\longmapsto\Lambda_\sigma,
\]
is real-analytic, and its differential is given by
\begin{equation}\label{eq:nd-derivative-form}
    \bigl\langle d\Lambda_\sigma(\eta)f,g\bigr\rangle
    =-\int_\Omega \eta\,\nabla u_{\sigma,f}\cdot\nabla u_{\sigma,g}\,dx,
    \qquad \eta\in L^\infty(\Omega),\;f,g\in L^2_\diamond(\partial\Omega).
\end{equation}
In particular, $\Lambda_\sigma$ and all its differentials are compact operators on
$L^2_\diamond(\partial\Omega)$.
\end{proposition}

\begin{proof}
The map $\sigma\mapsto\Lambda_\sigma$ is real-analytic from $L^\infty_+(\Omega)$ into
$\calL\bigl(H^{-1/2}_\diamond(\partial\Omega),H^{1/2}_\diamond(\partial\Omega)\bigr)$,
with the differential \eqref{eq:nd-derivative-form}; see
\cite[Theorem 3.3]{garde2022series}, \cite[Theorem A.2]{garde2020regularity} and
\cite{lechleiter2008newton}. Pre-composing with the continuous embedding
$L^2_\diamond(\partial\Omega)\hookrightarrow H^{-1/2}_\diamond(\partial\Omega)$ and
post-composing with the compact embedding
$H^{1/2}(\partial\Omega)\hookrightarrow L^2(\partial\Omega)$, which is compact because
$\partial\Omega$ is Lipschitz, is a bounded linear operation with values in
$\calK\bigl(L^2_\diamond(\partial\Omega)\bigr)$. Real-analyticity is preserved by
composition with bounded linear maps, and, since
$\calK\bigl(L^2_\diamond(\partial\Omega)\bigr)$ is closed in
$\calL\bigl(L^2_\diamond(\partial\Omega)\bigr)$, the differentials take values there too.
\end{proof}

\begin{proof}[Proof of \cref{thm:calderon-gaussian}]
Apply the Gaussian architecture (\Cref{sec:gaussian-sampling}) with
\[
    H=Z=L^2_\diamond(\partial\Omega),\qquad r=2,\qquad\psi=I_{L^2_\diamond(\partial\Omega)},
\]
and with $F(\sigma)$ equal to the bounded bilinear form 
\begin{equation*}
    \Lambda_\sigma(f,g)=\langle \Lambda_\sigma f,g\rangle_{L^2(\partial\Omega)}.
\end{equation*}
Analyticity follows from \cref{prop:nd-analytic}, and injectivity is assumed. The conclusion follows directly by \cref{thm:main}.
\end{proof}

The proof of \cref{thm:calderon-finite} is based on a truncation argument. The following derivation is modeled on the compactness mechanisms used in \cite{alberti2022infinite,alberti2022inverse}, but it is stated here in a form that simultaneously preserves injectivity of the forward map and injectivity of its differential.

\begin{lemma}\label[lemma]{lem:compact-secants}
Let $\calM\subset X$ be a finite-dimensional $C^1$ embedded manifold, let $U\subset\calM$ be open, let $K\subset U$ be compact, and let $G\in C^1(U;Y)$, where $Y$ is a Banach space. Then the normalized secant set
\[
    \calS_K(G)
    =\left\{\frac{G(x)-G(y)}{\norm{x-y}_X}:x,y\in K,\ x\neq y\right\}
\]
is relatively compact in $Y$. Moreover,
\[
    \calT_K(G)
    =\left\{dG_x(v):x\in K,\ v\in T_x\calM,\ \norm{v}_X=1\right\}
\]
is compact.

If $G$ is injective on $K$ and $dG_x$ is injective for every $x\in K$, then
\begin{equation}\label{eq:secant-lower-bound}
    0\notin\overline{\calS_K(G)}\cup\calT_K(G).
\end{equation}
\end{lemma}

\begin{proof}
Take a sequence of normalized secants $\frac{G(x_j)-G(y_j)}{\norm{x_j-y_j}_X}$ for $x_j,y_j\in K$. After passing to a subsequence,  $x_j,y_j$ converge to $x,y\in K$, respectively. If $x\neq y$, the convergence of $\frac{G(x_j)-G(y_j)}{\norm{x_j-y_j}_X}$ is immediate. If $x=y$, choose a $C^1$ chart $\varphi:V\subset\RR^d\to\calM$ around $x$, with $V$ convex. Write the endpoints as $x_j=\varphi(p_j)$ and $y_j=\varphi(q_j)$, where $p_j,q_j\to q=\varphi^{-1}(x)$, and pass to a subsequence such that
\[
    \frac{p_j-q_j}{|p_j-q_j|}\longrightarrow v,\qquad |v|=1.
\]
Set $h_j=|p_j-q_j|$ and $u_j=(p_j-q_j)/h_j\to v$. Since $V$ is convex, the
segment $\gamma_j(s)=q_j+s(p_j-q_j)$, $s\in[0,1]$, lies in $V$, and for any
$C^1$ map $\Phi$ from $V$ into a Banach space the fundamental theorem of calculus
gives
\[
    \frac{\Phi(p_j)-\Phi(q_j)}{h_j}=\int_0^1 d\Phi_{\gamma_j(s)}(u_j)\,ds .
\]
Hence
\[
    \Bigl\|\frac{\Phi(p_j)-\Phi(q_j)}{h_j}-d\Phi_{q}(v)\Bigr\|
    \le \sup_{s\in[0,1]}\|d\Phi_{\gamma_j(s)}-d\Phi_{q}\|
       +\|d\Phi_{q}\|\,|u_j-v| \longrightarrow 0,
\]
because $\gamma_j(s)\to q$ uniformly in $s$ and $d\Phi$ is continuous. Applying
this to $\Phi=G\circ\varphi$ and to $\Phi=\varphi$ (and using the continuity of the
norm in the latter case) yields
\[
    \frac{G(\varphi(p_j))-G(\varphi(q_j))}{|p_j-q_j|}
    \longrightarrow d(G\circ\varphi)_{q}(v),
\]
and
\[
    \frac{\norm{\varphi(p_j)-\varphi(q_j)}_X}{|p_j-q_j|}
    \longrightarrow \norm{d\varphi_q(v)}_X.
\]
Note that $\norm{d\varphi_q(v)}_X>0$ because $\varphi$ is a diffeomorphism onto its image.
Consequently,
\[
    \frac{G(\varphi(p_j))-G(\varphi(q_j))}
         {\norm{\varphi(p_j)-\varphi(q_j)}_X}
    \longrightarrow
    \frac{d(G\circ\varphi)_q(v)}{\norm{d\varphi_q(v)}_X},
\]
so the normalized secants have a convergent subsequence. Hence, $\mathcal S_K(G)$ is relatively compact in $Y$.

Let now $z_j=dG_{x_j}(v_j)$ with $x_j\in K$, $v_j\in T_{x_j}\calM$ and
$\norm{v_j}_X=1$. Since $K$ is compact, after passing to a subsequence we
have $x_j\to x\in K$. Choose a $C^1$ chart $\varphi\colon V\subset\RR^d\to\calM$
with $\varphi(p_0)=x$, and a closed ball $B\subset V$ centred at $p_0$. For
$j$ large we can write $x_j=\varphi(p_j)$ with $p_j\in B$ and $p_j\to p_0$, and,
since $T_{x_j}\calM=d\varphi_{p_j}(\RR^d)$, also $v_j=d\varphi_{p_j}(w_j)$ for
some $w_j\in\RR^d$. The function
$(p,u)\mapsto\norm{d\varphi_p(u)}_X$ is continuous on the compact set
$B\times S^{d-1}$ and positive there, because $d\varphi_p$ is injective.
Hence it has a minimum $c>0$, and by homogeneity
$\norm{d\varphi_p(w)}_X\ge c|w|$ for all $p\in B$ and $w\in\RR^d$. In
particular, $|w_j|\le\norm{v_j}_X/c=1/c$. By Bolzano--Weierstrass, after a further
subsequence $w_j\to w$. Continuity of $p\mapsto d\varphi_p$ and $|w_j|\le1/c$ give
\[
  \norm{v_j-d\varphi_{p_0}(w)}_X
  \le\norm{d\varphi_{p_j}-d\varphi_{p_0}}\,|w_j|
     +\norm{d\varphi_{p_0}}\,|w_j-w|\longrightarrow0.
\]
As a consequence, $v=d\varphi_{p_0}(w)\in T_x\calM$ and $\norm{v}_X=\lim_j\norm{v_j}_X=1$.
In the same way, continuity of $p\mapsto d(G\circ\varphi)_p$ gives
\[
  z_j=d(G\circ\varphi)_{p_j}(w_j)\longrightarrow d(G\circ\varphi)_{p_0}(w)
  =dG_x(v)\in\calT_K(G).
\]
Thus every sequence in $\calT_K(G)$ has a subsequence converging in
$\calT_K(G)$, so $\calT_K(G)$ is compact.

If a limit of secants with distinct limiting endpoints were zero, injectivity of $G$ on $K$ would fail. If the endpoints were equal, the preceding calculation shows that the limit is a nonzero normalized derivative, since $dG_x$ is injective. Thus $0\notin \overline{\calS_K(G)}$. The same injectivity implies that  $0\notin\calT_K(G)$.
\end{proof}

\begin{proposition}\label[proposition]{prop:truncation-preserves-injectivity}
In the setting of \cref{lem:compact-secants}, assume that $G(U)\subset Y_0$, where $Y_0$ is
a closed subspace of $Y$, and let $Q_N\in\calL(Y_0)$ be uniformly bounded with
$Q_Ny\to y$ in $Y_0$ for every $y\in Y_0$.  Assume that $G$ is injective on $K$ and $dG_x$ is
injective for every $x\in K$. Then, for all sufficiently large $N$,
\begin{enumerate}[label=\textup{(\roman*)},leftmargin=*]
    \item $Q_NG$ is injective on $K$; \label{prop i}
    \item $Q_N\,dG_x\colon T_x\calM\to Y_0$ is injective for every $x\in K$; \label{prop ii}
    \item $Q_NG$ is injective on an open neighbourhood of $K$ in $\calM$. \label{prop iii}
\end{enumerate}
\end{proposition}

\begin{proof}
First note that, since $Y_0$ is closed, one has
$dG_x(T_x\calM)\subset Y_0$ for every $x\in U$. By \cref{lem:compact-secants} and \eqref{eq:secant-lower-bound}, the set
\[
    C=\overline{\calS_K(G)}\cup\calT_K(G)
\]
is compact and does not contain zero. If $C=\varnothing$, the conclusions are immediate.
Otherwise, set
\[
    \delta=\min_{v\in C}\norm{v}_Y>0.
\]
The family $(Q_N)_{N\in\NN}$ is uniformly bounded, hence equicontinuous, so the pointwise
convergence $Q_N\to I$ is uniform on the compact set $C$ (see \cite{Schaefer1999-dc}).
Hence, for all sufficiently large $N$,
\[
    \sup_{v\in C}\norm{Q_Nv-v}_Y<\frac{\delta}{2},
\]
and therefore
\[
    \norm{Q_Nv}_Y\geq\frac\delta2>0,
    \qquad v\in C .
\]
Fix such an $N$. If $x,y\in K$ with $x\neq y$, the normalized secant
$v=\bigl(G(x)-G(y)\bigr)/\norm{x-y}_X$ belongs to $\calS_K(G)\subset C$, so the estimate
above gives $Q_Nv\neq0$ and hence $Q_NG(x)\neq Q_NG(y)$; this proves \ref{prop i}. If
$x\in K$ and $v\in T_x\calM$ has $\norm{v}_X=1$, then $dG_x(v)\in\calT_K(G)\subset C$, so
the same estimate gives $Q_N\,dG_x(v)\neq0$; since $dG_x$ and $Q_N$ are linear,
$Q_N\,dG_x$ is injective on $T_x\calM$, which proves \ref{prop ii}.

It remains to prove \ref{prop iii}. Set $G_N=Q_NG\in C^1(U;Y_0)$, whose differential at
$x$ is $Q_N\,dG_x$, and fix $x\in K$. By \ref{prop ii} the map $d(G_N)_x$ is injective,
and its range is finite-dimensional, hence complemented in $Y_0$; there is therefore
$L\in\calL(Y_0;\RR^d)$ such that $L\circ d(G_N)_x$ is invertible, and the inverse function
theorem applied to $L\circ G_N$ in a chart at $x$ shows that $G_N$ is injective on a
neighbourhood of $x$. If no neighbourhood of $K$ had the property asserted in
\ref{prop iii}, there would be $x_j\neq y_j$ in $U$ with
$\dist(x_j,K)+\dist(y_j,K)\to0$ and $G_N(x_j)=G_N(y_j)$. By compactness of $K$, up to
subsequences $x_j\to x$ and $y_j\to y$ with $x,y\in K$; continuity of $G_N$ and
\ref{prop i} force $x=y$, so that for $j$ large both points lie in a neighbourhood of $x$
on which $G_N$ is injective, a contradiction.
\end{proof}

\begin{proof}[Proof of \cref{thm:calderon-finite}]
Let $P_N$ be the orthogonal projection from $L^2_\diamond(\partial\Omega)$ onto $H_N$, see
\eqref{eq HN}, and set $Q_NT=P_NTP_N$ for $T\in\calK\bigl(L^2_\diamond(\partial\Omega)\bigr)$.
Then $\norm{Q_N}\leq1$, and $Q_NT\to T$ in norm for every compact $T$, since
$\norm{T-P_NTP_N}\leq\norm{(I-P_N)T}+\norm{T(I-P_N)}$ and both terms vanish in the limit
by \cite[Lemma 3]{alberti2022infinite}, $P_N$ being self-adjoint and strongly convergent
to the identity.

By \cref{prop:nd-analytic}, $G(\sigma)=\Lambda_\sigma$ and its differentials take values in
$\calK\bigl(L^2_\diamond(\partial\Omega)\bigr)$, so that \ref{prop i} and \ref{prop ii} of \cref{prop:truncation-preserves-injectivity} apply to $G$ on $K$ and to the operators
$Q_N$: for every sufficiently large $N$, the truncated map
$G_N(\sigma)=P_N\Lambda_\sigma P_N$ is injective on $K$, and
$d(G_N)_\sigma=Q_N\,d\Lambda_\sigma$ is injective on $T_\sigma\calM$ for every
$\sigma\in K$. By \ref{prop iii} of \cref{prop:truncation-preserves-injectivity}, $G_N$ is injective on
$U_N=A_N\cap\calM$ for some open $A_N\subset L^\infty_+(\Omega)$ containing $K$, because
$\calM$ carries the subspace topology of $L^\infty(\Omega)$.

By \cref{ex:manifold}, $A_N\cap\calM$ is an analytic model class of dimension $d$, and the
scalar map $(\sigma,f,g)\mapsto\langle\Lambda_\sigma f,g\rangle$ is real analytic on
$A_N\times H_N\times H_N$. Its values for all $(f,g)\in H_N\times H_N$ determine
$G_N(\sigma)$, so that injectivity of $G_N$ on $U_N$ gives the separation condition
\eqref{eq:analytic-separation}. \Cref{thm:main}, applied with sampling manifold
$H_N\times H_N$ and law $\rho_N\otimes\rho_N$, together with $K\subset U_N$, yields the
conclusion.
\end{proof}

\begin{proof}[Proof of \cref{cor:moving}]
We verify that $\calM$ is an analytic model class. Write
\[
    \mathcal H_\diamond
    =\left\{u\in H^1(\Omega):
        \int_{\partial\Omega}u\,ds=0\right\}.
\]
First of all we notice that $\Phi_v$ is a bi-Lipschitz homeomorphism of $\overline\Omega$ onto itself. Indeed, since $\Phi_v$ is a bijection of $\overline\Omega$ that fixes every point of $\partial\Omega$, it maps $\Omega$ onto $\Omega$. Moreover, on each moved triangle
$\Phi_v(\overline{P_l})$, the inverse $\Phi_v^{-1}$ is affine with Jacobian
matrix $Q_l(v)^{-1}$. Thus $\Phi_v$ and $\Phi_v^{-1}$ are continuous and affine
on each of finitely many triangles, hence Lipschitz. 

Since $\Phi_v$ is bi-Lipschitz and fixes the boundary, composition with $\Phi_v$
is an isomorphism of $\mathcal H_\diamond$ that preserves traces.
The change of variables $x=\Phi_v(y)$ transforms the conductivity bilinear form
into
\[
    \int_\Omega \sigma_t\nabla u\cdot\nabla w\,dx
    =
    \int_\Omega \bigl(A_t\nabla(u\circ\Phi_v)\bigr)
                       \cdot\nabla(w\circ\Phi_v)\,dy,
    \qquad
    A_t|_{P_l}
    =c_l\det \bigl(Q_l(v)\bigr)\,Q_l(v)^{-1}Q_l(v)^{-\mathsf T}.
\]
This is the classical transformation rule for conductivities under a change of
variables, which leaves the boundary measurements unchanged when the boundary is
fixed \cite{kohn1984identification, astala2005calderon,kohn2008cloaking, beretta2018reconstruction}. The entries of these matrices are rational functions of $t$ with nonvanishing
denominators on $U$. Thus
\[
    t\longmapsto A_t
    \quad\text{is real analytic into }
    L^\infty(\Omega;\RR^{2\times2}).
\]
Each $A_t$ is symmetric and uniformly elliptic, with ellipticity constants
locally uniform in $t$. Consequently,
\[
    B_t:\mathcal H_\diamond\longrightarrow\mathcal H_\diamond^*,
    \qquad
    (B_tu)(w)=\int_\Omega A_t\nabla u\cdot\nabla w\,dy,
\]
is invertible by the Lax--Milgram theorem and depends analytically on $t$ in
operator norm. Analyticity of inversion gives that
$t\mapsto B_t^{-1}$ is analytic; this Neumann series argument is the one used by
Calder\'on \cite{calderon1980inverse} to show that the boundary measurements
depend analytically on the conductivity.

Let $T:\mathcal H_\diamond\to L^2_\diamond(\partial\Omega)$ be the trace operator
and let $J=T^*$ be given by
\[
    J:L^2_\diamond(\partial\Omega)\to\mathcal H_\diamond^*,
    \qquad
    (Jf)(w)=\int_{\partial\Omega}fw\,ds.
\]
Neither the boundary forcing nor the solution trace changes under the
pullback, so
\[
    \Lambda_{\sigma_t}=TB_t^{-1}J.
\]
Hence
$t\mapsto[(f,g)\mapsto\langle\Lambda_{\sigma_t}f,g\rangle]$ is real analytic on
$U$. This verifies condition~\ref{ii} of \cref{def:model-class}.
Finally, planar uniqueness for positive bounded scalar conductivities
\cite{astala2006calderon}
gives injectivity on $\calM$. The conclusion follows from
\cref{thm:calderon-gaussian}.
\end{proof}

\section{Proof of the result on inverse scattering}\label{sec:proof-scattering}

\begin{proposition}\label[proposition]{prop:far-field-analytic}
The map
\[
    \mathcal A_k\times S^2\times S^2\longrightarrow\CC,
    \qquad
    (n,\widehat x,\theta)\longmapsto u_n^\infty(\widehat x,\theta),
\]
is real-analytic.
\end{proposition}

\begin{proof}
For fixed $n$, joint real analyticity in $(\widehat x,\theta)$ is proved in \cite{griesmaier2013note}; here we also include the dependence on $n$.
The map $n\mapsto V_n$ in \eqref{eq:volume-potential} is affine and continuous from $L^\infty(B;\CC)$ into $\calL(L^2(B))$. The set of invertible bounded operators is open, and inversion is analytic there. Therefore
\[
    n\longmapsto (I-V_n)^{-1}
\]
is analytic on $\mathcal A_k$. The map $\theta\mapsto e^{ik(\cdot)\cdot\theta}$ is real-analytic from $S^2$ into $L^2(B)$, so \eqref{eq:lippmann-schwinger} shows that $(n,\theta)\mapsto u_{n,\theta}$ is jointly real-analytic as an $L^2(B)$-valued map.

Similarly, $\widehat x\mapsto e^{-ik\widehat x\cdot(\cdot)}$ is real-analytic as an $L^2(B)$-valued map. Formula \eqref{eq:far-field} is obtained by composing these maps with the continuous trilinear map
\[
    L^2(B)\times L^\infty(B)\times L^2(B)\to\CC,\qquad(\phi,m,v)\longmapsto\int_B\phi(y)m(y)v(y)\,dy.
\]
The conclusion follows.
\end{proof}

\begin{proof}[Proof of \cref{thm:scattering}]
Use the finite-dimensional analytic sampling architecture (\Cref{sec:finite-analytic-sampling}) with
\[
    Z=S^2\times S^2,
    \qquad
    \mathfrak m(n,(\widehat x,\theta))=u_n^\infty(\widehat x,\theta).
\]
Joint analyticity follows from \cref{prop:far-field-analytic}. If $n_1\neq n_2$ in $\mathcal A_k\cap\calM$, injectivity of the full far-field map implies that
\[
    (\widehat x,\theta)\longmapsto
    u_{n_1}^\infty(\widehat x,\theta)-u_{n_2}^\infty(\widehat x,\theta)
\]
is not identically zero. Thus \eqref{eq:analytic-separation} holds, and \cref{thm:main} gives the result.
\end{proof}

\section*{Disclosure of AI use}

We used OpenAI's ChatGPT (5.5 Pro, 5.6 Pro, and Astra)
and Anthropic's Claude (Opus 4.6 and Opus 5) to assist with literature
review, grammar and English expression, and the checking and
simplification of proofs. These tools also helped us become familiar
with the theory of semianalytic and subanalytic sets. In particular,
Astra brought Sontag's 2002 paper~\cite{sontag2002experiments} to our
attention, providing a crucial reference for this work.

The problem formulation and the original proof strategy were developed
by the authors without AI assistance. Subsequent discussions with
ChatGPT helped us recognize that the argument, initially formulated
for embedded real analytic manifolds of dimension~$d$, also applies
in the more general setting of analytic model classes adopted in
this paper. AI tools further assisted in assessing whether particular
classes of examples satisfy the assumptions of this framework.

The authors retain full responsibility for the mathematical arguments
and results, the accuracy and completeness of the references, and
the final content of the manuscript.

\section*{Acknowledgements}
We thank Giacomo Traverso for his contributions in the early phase of the project. The research was supported in part by the MIUR Excellence Department Project awarded to Dipartimento di Matematica, Università di Genova, CUP D33C23001110001. Co-funded by the European Union (ERC, SAMPDE, 101041040). Views and opinions expressed are however those of the authors only and do not necessarily reflect those of the European Union or the European Research Council. Neither the European Union nor the granting authority can be held responsible for them.

\bibliographystyle{abbrvnat}
\bibliography{ref}

\end{document}